\documentclass[12pt]{amsart}
\usepackage{amsmath,amsthm,amsrefs}
\usepackage{amsbsy,amsfonts,amssymb,mathtools}
\usepackage{amsaddr, ytableau,kbordermatrix,mathtools}

\def\ds{\displaystyle}

\def\spo{spo(2m,n)}
\def\ov{\overline}
\def\c{\circ}

\def\ov{\overline}

\theoremstyle{plain}
\newtheorem{thm}{Theorem}[section]
\newtheorem{lem}[thm]{Lemma}

\newtheorem{cor}[thm]{Corollary}

\theoremstyle{definition}

\newtheorem{exa}[thm]{Example}

\newdimen\Squaresize \Squaresize=14pt
\newdimen\Thickness \Thickness=0.4pt

\def\Square#1{\hbox{\vrule width \Thickness
   \vbox to \Squaresize{\hrule height \Thickness\vss
      \hbox to \Squaresize{\hss#1\hss}
   \vss\hrule height\Thickness}
\unskip\vrule width \Thickness} \kern-\Thickness}

\def\Vsquare#1{\vbox{\Square{$#1$}}\kern-\Thickness}

\def\younglambda#1{
\vbox{\smallskip\offinterlineskip \halign{&\Vsquare{##}\cr #1}}}
\def\vy#1{\hskip2pt\vcenter{\younglambda{#1}}\hskip2pt}

\newcommand{\shiftbar}[2]{\ensuremath \raisebox{#1cm}{$\leftarrow\overline{#2}$}}
\newcommand{\shiftnobar}[2]{\ensuremath \raisebox{#1cm}{$\leftarrow{#2}$}}

\title[Orthosymplectic Cauchy identities]%
{Orthosymplectic Cauchy identities}

\author{Aalekh Patel, Harsh Patel, Anna Stokke}
\address{University of Winnipeg \\
Department of Mathematics and Statistics \\
Winnipeg, Manitoba \\
Canada  R3B 2E9}

\email{\tt a.stokke@uwinnipeg.ca}
\thanks{This research was supported by NSERC grant RGPIN-2018-05877.}

\begin{document}

\begin{abstract}  We give bijective proofs of orthosymplectic analogues of the Cauchy identity and dual Cauchy identity for orthosymplectic Schur functions.   To do so, we present two insertion algorithms; these are
orthosymplectic versions of  Berele's symplectic insertion algorithms, which were used by Sundaram to give bijective proofs of Cauchy identities for symplectic Schur functions.   \end{abstract}

\maketitle

\section{Introduction}

The Robinson-Schensted-Knuth (RSK) correspondence, which gives a bijection between generalized permutations and pairs of semistandard tableaux $(P,Q)$ of the same shape, yields a direct proof of the Cauchy identity for Schur functions \cite{knuth}.  Using an insertion algorithm for symplectic tableaux developed by Berele \cite{berele}, Sundaram proved a Cauchy identity for the characters of irreducible representations of symplectic groups in \cite{sundaram}.

Benkart, Shader and Ram worked with representations of orthosymplectic Lie superalgebras $spo(2m,n)$ in \cite{benkart} where they introduced $spo$-tableaux as hybrids of symplectic tableaux and semistandard tableaux in order to give combinatorial descriptions of characters of $\spo$-representations. 
They also gave algebraic proofs of various identities for orthosymplectic characters, including Jacobi-Trudi-type formulae and an analogue of the Cauchy identity.

In \cite{stokkevisentin} an orthosymplectic Jacobi-Trudi identity is proved combinatorially using lattice path arguments.  In \cite{stokkeosp}, the third author proved a Pieri rule, which describes the integer coefficients of the orthosymplectic characters that appear in the expansion of the product of two orthosymplectic characters, one of which is associated to a one-row tableau.  In the current paper, we give bijective proofs of orthosymplectic analogues of the Cauchy identity and dual Cauchy identity.

After establishing preliminaries in Section \ref{sec:prelim}, we recall the jeu de taquin algorithm for $spo$-tableaux from \cite{benkart} and then define an $spo$-insertion algorithm in Section \ref{sec:jdt}.  An $spo$-insertion algorithm is also given in \cite{benkart} and used in \cite{stokkeosp}; the one we describe here is slightly different, but both algorithms 
 are hybrids of regular RSK-insertion and the Berele insertion algorithm \cite{berele, sundaram}.  In Section \ref{sec:cauchy1}, we introduce the $spo$-correspondence,  which yields a bijection between a set $\mathcal{A}$ of two-line arrays defined in Section \ref{sec:prelim} and triples that consist of an $spo$-tableau, a semistandard tableau of the same shape and Burge two-line arrays.  This $spo$-correspondence is the main tool that we use to establish an orthosymplectic Cauchy identity.  In Section \ref{sec:dualortho} we give a dual $spo$-correspondence,  and use it to give a direct proof of an orthosymplectic analogue of the dual Cauchy identity.

\section{Preliminaries}\label{sec:prelim}
A {\em partition} of a positive integer $N$ is a $k$-tuple of  positive integers
$\lambda=(\lambda_1,\ldots,\lambda_k)$ with $\lambda_1 \geq \lambda_2 \geq \cdots \geq \lambda_k$ and $\vert \lambda \vert =\sum_{i=1}^k \lambda_i=N$. The {\em Young diagram} of shape $\lambda$ contains $N$
boxes in $k$ left-justified rows with $\lambda_i$ boxes in the $i$th row.  The {\em length} of the Young diagram of shape $\lambda$ is the number of rows in $\lambda$, denoted $\ell(\lambda)$.  The {\em conjugate} or {\em transpose} of $\lambda$ is the partition
$\lambda^t=(\lambda_1^t,\lambda_2^t,\ldots,\lambda_s^t)$ where
$\lambda_i^t$ is the number of boxes in the $i$th column of the
Young diagram of shape $\lambda$.  A partition is said to be {\em even} if the number of boxes in each row is even.

A $\lambda$-tableau  is obtained by filling the boxes of the Young diagram of shape $\lambda$ with entries from a set of positive integers
$\{1,2, \ldots, n\}$. A
$\lambda$-tableau is {\em semistandard} if the entries in the rows are weakly increasing from left to right and the entries in the
columns are strictly increasing from top to bottom.  Given a $\lambda$-tableau $T$, we will denote $sh(T)=\lambda$.

The \textit{weight} of a $\lambda$-tableau $T$ is a monomial in the variables $X = \{x_1, \ldots, x_n\}$ defined as $\displaystyle wt(T) = \prod_{i=1}^{n} x_i^{\alpha_i} $ where $\alpha_i $ denotes the number of appearances of the entry $i$ in $T$.
The \textit{Schur function} corresponding to $\lambda$ is defined as $\displaystyle s_\lambda(X) = \sum_{T}^{} wt(T) $, where  the sum runs over the semistandard $\lambda$-tableaux with entries in $\{1,\ldots,n\}$.

Given partitions $\lambda$ and $\mu$, we say that $\mu \subseteq \lambda $ if $\mu_i \le \lambda_i $, for $i \ge 1 $. A \textit{skew diagram} of shape $\lambda/\mu $ is obtained by removing the Young diagram of shape $\mu$ from the top-left corner of the Young diagram $\lambda$. A \textit{skew tableau} is obtained by filling a skew diagram with the entries from a set $\{1, 2, \ldots, n\}$. A \textit{semistandard skew tableau} is a skew tableau in which the entries in the rows weakly increase from left to right and the entries in the columns strictly increase from top to bottom. The weight of a skew tableau and \textit{skew Schur function} are defined similar to the above.

Let $B_0=\{1,\ov{1},2,\ov{2},\dots,m,\ov{m}\}$ with the ordering $1< \ov{1} < 2 < \ov{2} < \cdots < m < \ov{m}$
and suppose that $\lambda$ is a partition with $\ell(\lambda)\leq m$.   A {\em semistandard symplectic $\lambda$-tableau}   has entries in $B_0$, is semistandard with respect to the above ordering, and satisfies the {\em symplectic condition}, which requires that the entries
in the $i$th row of $T$ are greater than or equal to $i$, for each
$1\leq i \leq m$.  If an entry $j$ appears in row $i$ where $j<i$, we refer to this as a {\em symplectic violation}.

The weight
of a symplectic $\lambda$-tableau $T$   is a monomial in the entries $\overline{X}=\{x_1,x_1^{-1},\ldots,x_m,x_m^{-1}\}$ and is given by
$\mathrm{wt}(T)=\prod_{i=1}^m x_i^{a_i-a_{\ov{i}}} $  where $a_i$ (respectively $a_{\overline{i}}$) is equal to the number of entries equal to $i$ (respectively $\overline{i}$)  in $T$.  The {\em symplectic Schur function} is defined as
$$sp_{\lambda}(\overline{X})=\sum_{T} \mbox{wt}(T),$$ where the
sum runs over the semistandard symplectic $\lambda$-tableaux $T$ with entries in $\overline{X}$. 



Let $B_1=\{1^\circ,2^\circ,\ldots,n^\circ\}$ and order $B_0 \cup B_1$ as follows:
\[ 1 < \ov1 < 2 < \ov2 < \cdots < m < \ov{m} < 1^\circ < 2^\circ < \cdots < n^\circ. \]

An {\em $spo(2m,n)$-tableau} $T$ of shape $\lambda$ (or an $spo$-tableau) is a filling of the Young diagram of shape $\lambda$ with entries from $B_0 \cup B_1$ such that the portion $S$ of $T$ that contains entries only from $B_0$ is a semistandard symplectic tableau.  Furthermore, the skew tableau formed by removing $S$ from $T$, which contains only entries from $B_1$, is strictly increasing across the rows from left to right and weakly increasing down columns from top to bottom.

Let $\mathcal{T}=\{t_1,\ldots,t_n\}$ and $Z=\overline{X} \ \cup \ \mathcal{T}$.
The {\em orthosymplectic character} corresponding to $\lambda$ is given by
\[ spo_{\lambda}(Z)=\sum_{\substack{\mu \subseteq \lambda \\ \ell(\mu) \leq m}} sp_{\mu}(\overline{X})s_{\lambda^t /\mu^t}(\mathcal{T}), \]
where $sp_\mu(\overline{X})$ is a symplectic Schur function in $\overline{X}$ and $s_{\lambda^t/ \mu^t}(Y)$ is a skew Schur polynomial in $\mathcal{T}$.

To define an orthosymplectic character in terms of $spo$-tableaux, define the weight of an $spo$-tableau $T$ by replacing each entry $i \in \{1,\ldots,m\}$ in $T$ with $x_i$,
each $\overline{i} \in \{\ov{1},\ldots,\ov{m}\}$ in $T$ with $x_i^{-1}$ and each $i^\circ \in \{1^\c, \ldots,n^\c\}$ in $T$ with $t_i$ and
let  $\mbox{wt}(T)$ be the product of these variables.
We have $spo_{\lambda}(Z)=\sum_T \mbox{wt}(T)$,
where the sum is over the $spo$-tableaux of shape $\lambda$ with entries from $Z$ \cite[Theorem 5.1]{benkart}.


\begin{exa}Let $\lambda = (2)$, $m = 1$ and $n = 2$. Then $\overline{X} = \{x_1, x^{-1}_1\}$, $\mathcal{T}= \{t_1, t_2\}$.   The spo-tableaux are 
$$\vy{1 & 1\cr}, \ \vy{1 & \ov{1} \cr}, \ \vy{1 & 1^\circ \cr}, \ \vy{1 & 2^\circ \cr}, \ \vy{\ov{1} & \ov{1} \cr}, \ \vy{\ov{1} & 1^\circ\cr}, \ \vy{\ov{1} & 2^\circ\cr}, \ \vy{1^\circ & 2^\circ \cr},$$
and
\begin{align*}
spo_\lambda(Z) = x_1^2 + 1 + x_1t_1 + x_1t_2 + x_1^{-2} + x_1^{-1}t_1 + x_1^{-1}t_2 + t_1t_2
.\end{align*}
\end{exa}

\bigskip

 Let $q$ be a positive integer and consider the set of two-line arrays  $$ \pi = 
\begin{pmatrix}
a_1 & a_2  & \dots & a_t\\
b_1 & b_2 & \dots & b_t
\end{pmatrix}$$
such that $a_i \in \{1,\ldots,q\}$ and $b_i \in B_0 \cup B_1$.
Let $\mathcal{A}$ denote the set of such arrays that satisfy the following properties:
\begin{enumerate}
\item  $a_i \leq a_{i+1}$  and if $a_i =a_{i+1}$ then $b_i \leq b_{i+1}$;
\item any column $\ds \binom{i}{j}$ with $j \in B_1$ occurs at most once in $\pi$.
\end{enumerate}

\noindent Define $\mathcal{A}^*$ to be the set of arrays that satisfy:
\begin{enumerate}
\item $a_i \leq a_{i+1}$  and if $a_i =a_{i+1}$ then $b_i \geq b_{i+1}$;
\item any column $\ds \binom{i}{j}$ with $j \in B_0$ occurs at most once in $\pi$.
\end{enumerate}

\noindent A {\em Burge two-line array} $\pi$ has $a_i \in \{1,\ldots,q\}$, $b_i \in \{1,\ldots,q-1\}$ and satisfies
\begin{enumerate}
\item  $a_i \leq a_{i+1}$  and if $a_i =a_{i+1}$ then $b_i \leq b_{i+1}$;
\item $a_i > b_i$ for $1 \leq i \leq t$.
\end{enumerate}

\noindent A {\em dual Burge two-line array} $\pi$ has both $a_i,b_i \in \{1,\ldots,q\}$ and satisfies
\begin{enumerate}
\item  $a_i \leq a_{i+1}$  and if $a_i =a_{i+1}$ then $b_i \geq b_{i+1}$;
\item $a_i \geq b_i$ for $1 \leq i \leq t$.
\end{enumerate}

\begin{exa}
Let
$$ \ds \pi_1= 
\left(\begin{array}{llll}
1& 1 &  2& 2\\
2  &3^\circ & \ov{3} & \ov{3} \cr\end{array}\right), \ 
 \pi_2=\left(\begin{array}{llll}
1&1 & 2& 2\\
3^\circ & \ov{2} & 2^\circ & 2^\circ \cr \end{array}\right).$$
Then $\pi_1 \in \mathcal{A}, \pi_1 \notin \mathcal{A}^*, \pi_2 \in \mathcal{A}^*, \pi_2 \notin \mathcal{A}$. \end{exa}

\section{Orthosymplectic jeu-de-taquin and insertion}\label{sec:jdt}

A {\em punctured} $spo$-tableau is an $spo$-tableau with an entry removed from one of its boxes.  The jeu-de-taquin algorithm for $spo$-tableaux moves an empty box in the first column of a 
punctured $spo$-tableau through the tableau using {\em forward slides}.  A forward slide is defined as follows:

$$\label{jdtmoves}\vy{ & x \cr y \cr } \rightarrow
\left\{\begin{array}{ll}
\vy{x &  \cr  y \cr} & \text{if $x < y$ or $x=y$ and $x,y \in B_1$} \\[5mm]
\vy{y & x \cr  \cr} & \text{if $y<x $ or $x=y$ and $x,y \in B_0$.}
\end{array}\right.\ \ \ 
$$
 
The result of performing a forward slide on a punctured $spo$-tableau is a punctured $spo$-tableau and  forward slides can be continued, beginning with an empty box in the first column, until the empty box lands in an outer corner \cite[Lemmas 5.2 \& 5.3]{benkart}, which means there are no boxes directly below it or to its immediate right.  The procedure is also reversible and the slides that reverse the procedure are called {\em reverse slides}.

Ordinary Robinson-Schensted-Knuth (RSK) {\em row insertion} for semistandard 
$\lambda$-tableaux gives an algorithm for inserting a letter into a semistandard tableau and returns a new semistandard tableau with one more box (see \cite{fulton, sagan, schensted, stanley}).  The procedure is as follows for a semistandard $\lambda$-tableau $T$ and a positive integer $x$:
if $x$ is greater than or equal to all entries in the first row, add $x$ in a new box at the end of the first row, which yields a new semistandard tableau $T \leftarrow x$.  Otherwise, remove the left-most entry in the first row that is strictly larger than $x$, place  $x$ in that box, and bump the displaced entry into the second row by repeating the procedure with the displaced positive integer in the second row so that it is placed either at the end of the second row or replaces an entry in the second row.  Continue this process for successive rows until a new box is created  at the end of a row or at the bottom of the tableau, forming a new row.  The final step gives  a new semistandard tableau $S=T \leftarrow x$  and the added box belongs to an  {\em outer corner} of the Young diagram.  The row insertion algorithm can be reversed; starting with a box in an outside corner, the entry $x$ in that box replaces $x^\prime$ in the row above it, where $x^\prime$ is the rightmost entry in the row with $x^\prime < x$ and the procedure continues until an entry is bumped out of the first row.

Ordinary RSK-column insertion is similar except that $x$ is added to the end of  the first column if it is larger than or equal to all entries in the column, and otherwise replaces the smallest entry $y$ that is strictly larger than it in the first column and bumps $y$ into the second column.  In this case the process continues until a new box is created at the end of a column or until a new column is formed with one entry, with the new box necessarily appearing in an outer corner.   

We will also use {\em dual row insertion} of a letter $x$ into a tableau $T$ where the transpose $T^t$ is semistandard.   In this case, $x$ is appended to the end of the first row of $T$ if $x$ is larger than every entry in the row.  Otherwise, $x$ replaces the leftmost entry $x^\prime$  in the first row that satisfies $x^\prime \geq x$.  The entry $x^\prime$ is then bumped into the second row of the tableau, and the procedure repeats until a new box is created at the end of a row or at the bottom of the tableau.  This results in a new tableau $S=x \rightarrow T$ such that $S^t$ is semistandard.

We now define an insertion algorithm for spo-tableaux.  Our algorithm is slightly different than the insertion algorithm  in \cite{benkart} and \cite{stokkeosp}, where column insertion is applied to elements from $B_1$.   The algorithm we describe here instead uses dual row insertion for elements from $B_1$, which ensures that the second tableau in the triple we define in Section \ref{sec:cauchy1} is semistandard.  To insert an entry $x \in B_0 \cup B_1$ into an $spo$-tableau $T$, begin by considering the first row.  
\begin{enumerate}
\item If $x \in B_0$ and $x$ is larger than or equal to every entry in the first row or if $x \in B_1$ and $x$ is strictly larger than every entry in the first row, append $x$ to the end of the first row, giving $T \leftarrow x$. 
\item  Otherwise, if $x \in B_1$, find the left-most entry $x^\prime$ in the first row  with $x^\prime \geq x$ and replace this entry with $x$, bumping $x^\prime$.  If $x \in B_0$ find the left-most entry $x^\prime$ in the first row with $x^\prime > x$, replace $x^\prime$ with $x$, bumping $x^\prime$ from the first row.  
\item Insert $x^\prime$ into the second row by repeating the procedure and continue for successive rows unless a symplectic violation occurs.  
\item If no symplectic violation is introduced at any point in the process, the procedure ends when a new box is created at the end of a row or at the bottom of the tableau.  The resulting tableau is $T \leftarrow x$.
\item On the other hand, if a symplectic violation is introduced at some point in the procedure, then for some $i$,  an $\overline{i}$ is bumped out of row $i$ and into row $i+1$ by an $i$.  The Berele insertion algorithm is applied at this point:  find the least $i$ for which a violation occurs and, at this stage, instead of replacing the $\overline{i}$ with an $i$, remove the $\overline{i}$ from the tableau.   Apply jeu de taquin forward slides to the empty box until it lands in an outer corner, giving an $spo$-tableau $T \leftarrow x$ that has one less box than $T$.  
\end{enumerate}

\begin{lem}\label{firstlemma} Let $T$ be an $spo$-tableau and let $b \in B_0 \cup B_1$.  Then $T \leftarrow b$ is an $spo$-tableau.\end{lem}
\begin{proof}
The design of the algorithm preserves the row conditions.  Suppose an entry $y$ bumps an entry $x$ from a given row  and $x \in B_0$; if there is an entry $z$ below $x$ then  $z>x$ so insertion of $x$ into the subsequent row either places $x$ in the same column or in a column to the left, which necessarily has entries above it that are strictly less than $x$ since $y<x$.  On the other hand, if $x \in B_1$ and there is an entry $z$ below it, then $x \leq z$ and insertion of $x$ into the subsequent row places $x$ in the same column or to the left, with entries above it weakly less than $x$.

If a symplectic violation occurs at some point in the insertion process,  then $T \leftarrow b$ is an $spo$-tableau since the output of the jeu de taquin process is an $spo$-tableau.
\end{proof}

\begin{exa}  This example illustrates the insertion process when the entry $\overline{1}$ is inserted into the $spo$-tableau $T$.
$$T=\vy{ 1 &  \ov{1} & {\bf 2} & 1^\circ \cr 2 & \ov{2} & 3 & 2^\circ  \cr 4 & 4  & 2^\circ  \cr  5 \cr} \shiftbar{.62}{1}\ \ \ \ \  \vy{ 1 &  \ov{1} & \ov{1} & 1^\circ \cr 2 & {\bf \ov{2}} & 3 & 2^\circ  \cr 4 & 4  & 2^\circ  \cr  5 \cr} \shiftnobar{.2}{2} \ \ \ \ \  
\vy{ 1 &  \ov{1} & \ov{1} & 1^\circ \cr 2 &  & 3 & 2^\circ  \cr 4 & 4  & 2^\circ  \cr  5 \cr} \ \ \ \ \  \vy{ 1 &  \ov{1} & \ov{1} & 1^\circ \cr 2 & 3 &  2^\circ  \cr 4 & 4  & 2^\circ  \cr  5 \cr}
$$

\end{exa}

\section{Establishing the orthosymplectic Cauchy identity}\label{sec:cauchy1}
In this section we give a bijection, which we call the $spo$-correspondence, between the set $\mathcal{A}$ of two-line arrays defined in Section \ref{sec:prelim} and the set of triples $(\tilde{P}_\lambda, P_\lambda, L)$, where $\tilde{P}_\lambda$ is an $spo$-tableau, $P_\lambda$ a semistandard $\lambda$-tableau and $L$ is a Burge two-line array.  Applying the bijection between Burge two-line arrays and semistandard tableaux $Q_\beta$, with even length columns as in \cite{sundaram} then gives a bijection between $\mathcal{A}$ and triples $(\tilde{P}_\lambda, P_\lambda, Q_\beta)$, which establishes an orthosymplectic analogue of the Cauchy identity.

Given a two-line array $\pi \in \mathcal{A}$ with $s$ columns, we perform the steps described below to obtain a triple $(\tilde{P}_\lambda, P_\lambda, L)$.

\begin{enumerate}
	\item Set $\tilde{P}_{0} = P_{0} = L_{0} = \phi$.
	\item Working from left to right, and beginning with the first column, consider the $k^{th}$ column in $\pi$, with top entry $i$ and bottom entry $j$. To obtain $(\tilde{P}_{k}, P_{k}, L_k) $ from $(\tilde{P}_{k-1}, P_{k-1}, L_{k-1}) $, apply the $spo$-insertion algorithm to get $\tilde{P}_k=\tilde{P}_{k-1}\leftarrow j $. This operation either adds or deletes exactly one box in $\tilde{P}_{k-1}$. 
\begin{enumerate}	
\item If a box is added by the insertion of $j$, form $P_{k}$ by adding a new box to $P_{k-1}$ containing $i$ at the same position where a new box was added to $\tilde{P}_{k-1}$ to get $\tilde{P}_{k}$. This creates $P_k $ with the same 
shape as that of $\tilde{P}_{k} $. In this case, let $L_k = L_{k-1} $.

\item If $\tilde{P}_k$ has one less box than $\tilde{P}_{k-1}$, the extra box in $\tilde{P}_{k-1}$ occurs in an outer corner.  Locate the position of the extra box in $P_{k-1}$ and perform reverse column bumping on the element that belongs to the extra box to get $P_k $, which has the same shape as $\tilde{P}_k $.  This will bump an entry $a$ from the first column.  Add the column $\binom{i}{a}$ on the right of $L_{k-1} $ to obtain $L_k $.
\end{enumerate}

	\item Let $\tilde{P} =\tilde{P}_{s} $, $P=P_s $ and $L=L_s$, obtaining the triple $(\tilde{P}_\lambda,P_\lambda,L)$.
\end{enumerate}

\begin{exa}\label{firstmap} Let
$$\pi = 
\begin{pmatrix*}[l]
2 & 2 & 2 & 3 & 3 & 4 & 4 & 4 & 4 \\
\overline{1} & 1^\circ & 2^\circ & 1 & \overline{1} & 1 & 1 & 1^\circ & 2^\circ
\end{pmatrix*}
$$
The tableaux and two-line arrays obtained at each step of the algorithm are listed below.\\
\begin{align*}
\tilde{P_1} =\vy{
\overline{1} \cr}
\qquad
P_1 = 
\vy{2\cr}
 \qquad
L_1 = \phi
\end{align*}
\begin{align*}
\tilde{P_2} = 
\vy{ \overline{1} & 1^\circ
\cr} \qquad
P_2 = 
\vy{2 & 2\cr}
 \qquad
L_2 = \phi
\end{align*}
\begin{align*}
\tilde{P_3} = 
\vy{
\overline{1} & 1^\circ & 2^\circ \cr}
\qquad
P_3 = 
\vy{
2 & 2 & 2 \cr} \qquad
L_3 = \phi
\end{align*}
\begin{align*}
\tilde{P_4} = 
\vy{
1^\circ & 2^\circ\cr }\qquad
P_4 = 
\vy{
2 & 2 \cr}
 \qquad
L_4 = 
\begin{pmatrix*}[l]
3 \\
2
\end{pmatrix*}
\end{align*}

\begin{align*}
\tilde{P_5} = 
\vy{
\overline{1} & 2^\circ \cr
1^\circ \cr}  \qquad
P_5 = 
\vy{
2 & 2 \cr
3\cr }
 \qquad
L_5 = 
\begin{pmatrix*}[l]
3 \\
2
\end{pmatrix*}
\end{align*}

\begin{align*}
\tilde{P_6} = 
\vy{
1^\circ & 2^\circ \cr
} \qquad
P_6 = 
\vy{
2 & 2
\cr} \qquad
L_6 = 
\begin{pmatrix*}[l]
3 & 4\\
2 & 3
\end{pmatrix*}
\end{align*}

\begin{align*}
\tilde{P_7} = 
\vy{
1 & 2^\circ \cr
1^\circ \cr} \qquad
P_7 = 
\vy{
2 & 2 \cr
4 \cr} \qquad
L_7 = 
\begin{pmatrix*}[l]
3 & 4\\
2 & 3
\end{pmatrix*}
\end{align*}

\begin{align*}
\tilde{P_8} = 
\vy{
1 & 1^\circ \cr
1^\circ & 2^\circ
\cr} \qquad
P_8 = 
\vy{
2 & 2\cr
4 & 4 \cr
}\qquad
L_8 = 
\begin{pmatrix*}[l]
3 & 4\\
2 & 3
\end{pmatrix*}
\end{align*}

\begin{align*}
\tilde{P_9} = \vy{1 & 1^\circ & 2^\circ \cr
1^\circ & 2^\circ \cr}
\qquad
P_9 = 
\vy{
2 & 2 & 4 \cr
4 & 4 \cr} \qquad
L_9 = 
\begin{pmatrix*}[l]
3 & 4\\
2 & 3
\end{pmatrix*}
\end{align*}
The image of $\pi$ is $(\tilde{P_9}, P_9, L_9)$.

\end{exa}

\begin{lem}\label{strictlyrightlem} Suppose that $T$ is an $spo$-tableau and that $x, x^\prime \in B_0 \cup B_1$ with $x \leq x^\prime$ if $x, x^\prime \in B_0$ and $x<x^\prime$ otherwise.  Suppose that $T \leftarrow x$ has one more box $B$ than $T$ and that $(T \leftarrow x)\leftarrow x^\prime$ has one more box $B^\prime$ than $T \leftarrow x$.  Then $B^\prime$ is strictly right of $B$.\end{lem}
\begin{proof}   Since each insertion adds a box to the starting tableau, the $spo$-insertion process is similar in this case to regular row insertion, except that an entry $y \in B_1$ may bump and replace an identical entry $y$ from a row during the process.  Since $x<x^\prime$ if either element belongs to $B_1$, the proof involves a slight adjustment of the one for regular insertion (see, for example \cite{fulton}). If $x$ bumps an entry $y$ from the first row and $x^\prime$ subsequently bumps $y^\prime$ from the first row then, since $x\leq x^\prime$, $y^\prime$ must belong to a box that is strictly right of $y$ in the first row of $T$.  Furthermore, if $y,y^\prime \in B_0$ then $y\leq y^\prime$ and otherwise $y<y^\prime$ so the same reasoning applies to subsequent rows, which inductively gives the result. \end{proof}

\noindent The following  two lemmas are needed to establish Theorem \ref{bijectionthm} and are orthosymplectic versions of  \cite[Lemmas 3.3 and 3.2]{sundaram}.    We omit the proofs since they are essentially the same as those given in \cite{sundaram}, adjusted slightly to accommodate jeu de taquin slides that involve entries from $B_1$.

\begin{lem}\label{taquinpathlem} Let $T$ be an $spo$-tableau and let $x, x^\prime \in B_0$ with $x \leq x^\prime$.  Suppose that $spo$-insertion of $x$ into $T$ causes a cancellation and $spo$-insertion of $x^\prime$ into $T \leftarrow x$ also causes a cancellation.  Then the jeu de taquin path of the empty box created by inserting $x$ ends in the same row or higher than the jeu de taquin path of the empty box created by insertion of $x^\prime$ into $T \leftarrow x$. \end{lem}

\begin{lem}\label{initialstrip} Suppose that $T$ is an $spo$-tableau and let $x,x^\prime \in B_0$ with $x\leq x^\prime$.  If $spo$-insertion of $x^\prime$ into $T \leftarrow x$ causes a cancellation then so did the prior $spo$-insertion of $x$ into $T$. \end{lem}

\begin{thm} \label{bijectionthm}The $spo$-correspondence yields a triple $(\tilde{P}_\lambda,P_\lambda,L)$, where $\tilde{P}_\lambda$ is an $spo$-tableau, $P_\lambda$ is semistandard and $L$ is a Burge two-line array. \end{thm}

\begin{proof} 
Lemma \ref{firstlemma} guarantees that $\tilde{P}_\lambda$ is an $spo$-tableau.  If $P_{k-1}$ is semistandard and $P_k$ has one less box than $P_{k-1}$, then $P_k$ is also semistandard since reverse column insertion preserves semistandard properties.  If a box is added to $P_{k-1}$ to form $P_k$, then it is added at an outer corner, since the corresponding box added to $\tilde{P}_{k-1}$ to form $\tilde{P}_k$ is at the same outer corner.  
Since the entries in the top row of  $\pi$ are weakly increasing, the rows and columns of the associated $P_k$ are weakly increasing.  Further, whenever $a_i$ and $a_{i+1}$ in the top row of $\pi$ are equal, the corresponding bottom entries $b_i$ and $b_{i+1}$ satisfy $b_i \leq b_{i+1}$ and  $b_i \neq b_{i+1}$ if $b_i, b_{i+1} \in B_1$, so it follows from Lemma \ref{strictlyrightlem} that the columns of $P_k$ are strictly increasing.

The proof that $L$ is a two-line Burge array is the same as in the symplectic case, which is covered in \cite{sundaram}.  Suppose that $L=\begin{pmatrix*}[l] i_1 & i_2 & \cdots & i_r \cr j_1 & j_2 & \cdots & j_r \cr \end{pmatrix*}.$
The entries in the top row of $L$ are weakly increasing by construction. Any entry $i_k$ in the top row of $L$ records the top entry in $\pi$ that corresponds to a cancellation when the entry below it in $\pi$ is inserted into some $\tilde{P}_t$.  The entry $j_k$ belongs to $P_t$ so it appears to the left of $i_k$ (so to the left of column $t$) in the top row of $\pi$ so $j_k \leq i_k$. If $i_k = j_k=a$ then there is at least one additional entry $a$ in the top row of $\pi$ left of column $t$ and, since $j_k=a$, $a$ bumps out of $P_t$.  By Lemma \ref{initialstrip} all entries in the bottom row of  $\pi$ to the left of column $t$ for which there is an $a$ in the top row must cause a cancellation so there cannot be any entries equal to $a$ in $P_t$ so $j_k <i_k$.

Finally, suppose that $i_k=i_{k+1}=a$ and let $b$ and $b^\prime$ denote the entries below $i_k$ and $i_{k+1}$ in $\pi$ respectively.  Both entries cause cancellations and by Lemma \ref{taquinpathlem}, the jeu de taquin path for $b^\prime$ is in a weakly lower row than that of $b$.  Then the corresponding column bumping path in $P_{t+1}$, which bumps $j_{k+1}$ out of the first column begins in a row weakly lower than the column bumping path in $P_t$, which bumps $j_k$ out of the first column.  By a property of column insertion \cite{schensted}, $j_k \leq j_{k+1}$.  

\end{proof}
To show that the $spo$-correspondence is invertible, we describe an algorithm below that maps a triple $(\tilde{P}_\lambda, P_\lambda, L)$, where $\tilde{P}_\lambda$ is an $spo$-tableau, $P_\lambda$ a semistandard tableau and $L$ a Burge two-line array to a two-line array $\pi \in \mathcal{A}$.  Given a two-line array $\pi$, the $spo$-correspondence produces  a triple  $(\tilde{P}_\lambda, P_\lambda, L)$ where the first row of  $L$ records entries from the top row of $\pi$ that correspond to cancellations, while the entries in $P_k$ are entries from the top row of $\pi$ that correspond to additions at each step.  If an entry appears in both  $P_k$  and in $L_k$, Lemma \ref{initialstrip} tells us that  the cancellations occurred prior to the additions.   Lemma \ref{strictlyrightlem} ensures that the order in which elements from the bottom row of $\pi$ were inserted into $\tilde{P}_k$ can be determined from $P_k$.  It thus follows that the map we now describe is the inverse of the above map.

\begin{enumerate}
	\item Set $\tilde{P_s} = \tilde{P}_\lambda $, $P_s = P_\lambda $, $L_s = L $ and ${\pi}_{s} = \phi $, where $n$ is the total number of entries in $P_\lambda $ and $L$ combined.
	\item Beginning with $k=s$,  to obtain $(\tilde{P}_{k-1},P_{k-1},L_{k-1}) $ and ${\pi}_{k-1} $ from $(\tilde{P}_{k}, P_{k}, L_{k})$ and ${\pi}_{k}$, find the largest entry $i$ among all the entries in $P_{k} $ and top row of $L_{k} $. 
	\begin{enumerate}
		\item If  $i$ belongs to $P_{k}$, then
		select the rightmost entry $i$ in $P_{k} $ and delete the box containing $i$ to get $P_{k-1} $. Since $\tilde{P}_{k} $ and $P_{k} $ are of the same shape, $\tilde{P}_k$ and $P_{k-1} $ differ by exactly one box. Perform reverse $spo$-insertion on the entry in the extra box in $\tilde{P}_k $ to obtain $\tilde{P}_{k-1} $ and suppose this bumps $j$ from the first row.  Prepend the column $ \binom{i}{j}$ to ${\pi}_{k} $ to get ${\pi}_{k-1} $ and let  $L_{k-1} = L_k. $
		\item If $i$ does not belong to $P_k$ (so only belongs to $L_{k} $),
		then select the right-most column $ \binom{i}{a}$ in $L_k$ and delete this column from $L_{k} $ to get $L_{k-1} $. Then column insert $a$ into $P_{k} $ to get $P_{k-1} $. This creates a new box and since $\tilde{P}_{k} $ and $P_{k}$ have the same shape, $\tilde{P}_{k} $ and $P_{k-1} $ differ by exactly one box.  Add an empty box to $\tilde{P}_{k} $ in this location and perform reverse jeu de taquin to slide the empty box to the highest possible row $r$ in the first column. Replace the hole with the row number $r$, and replace the rightmost $r$ in that row with $\overline{r}$. Note that there is at least one $r$ in the row $r$, since we placed  an entry $r$ in the empty box. Bump the entry $r$ into the row above it and perform reverse row bumping until some entry $j$ gets bumped out of the first row to obtain $\tilde{P}_{k-1} $. Prepend column $\binom{i}{j}$ to ${\pi}_{k} $ to obtain ${\pi}_{k-1} $.
	\end{enumerate}
	\item Let $\pi=\pi_0$.
\end{enumerate}

\begin{exa}

We will apply the inverse algorithm to the triple obtained in Example \ref{firstmap}. Let
$$ (\tilde{P}_\lambda, P_\lambda, L ) = \Bigg ( \quad
\vy{
1 & 1^\circ & 2^\circ \cr
1^\circ & 2^\circ \cr
}, \quad
\vy{
2 & 2 & 4 \cr
4 & 4\cr},
 \quad
\begin{pmatrix*}[l]
3 & 4 \\
2 & 3
\end{pmatrix*}
\Bigg )
$$
We obtain the following tableaux and two-line arrays.
\begin{align*}
\tilde{P}_9 = 
\vy{
1 & 1^\circ & 2^\circ \cr
1^\circ & 2^\circ \cr}
 \quad
P_9 = 
\vy{
2 & 2 & 4 \cr
4 & 4 \cr} \quad
L_9 = 
\begin{pmatrix*}[l]
3 & 4 \\
2 & 3
\end{pmatrix*} \quad
{\pi}_9 = \phi
\end{align*}

\begin{align*}
\tilde{P}_8 = 
\vy{
1 & 1^\circ \cr
1^\circ & 2^\circ \cr
} \quad
P_8 = 
\vy{
2 & 2 \cr
4 & 4\cr } \quad
L_8 = 
\begin{pmatrix*}[l]
3 & 4 \\
2 & 3
\end{pmatrix*} \quad
{\pi}_8 =
\begin{pmatrix*}[l]
4 \\
2^\circ
\end{pmatrix*}
\end{align*}

\begin{align*}
\tilde{P}_7 = 
\vy{
1 & 2^\circ \cr
1^\circ \cr
}\quad
P_7 =
\vy{
2 & 2 \cr
4 \cr}
 \quad
L_7 = 
\begin{pmatrix*}[l]
3 & 4 \\
2 & 3
\end{pmatrix*} \quad
{\pi}_7 =
\begin{pmatrix*}[l]
4 & 4 \\
1^\circ & 2^\circ
\end{pmatrix*}
\end{align*}

\begin{align*}
\tilde{P}_6 = 
\vy{
1^\circ & 2^\circ \cr
} \quad
P_6 =
\vy{
2 & 2 
\cr}\quad
L_6 = 
\begin{pmatrix*}[l]
3 & 4 \\
2 & 3
\end{pmatrix*} \quad
{\pi}_6 =
\begin{pmatrix*}[l]
4 & 4 & 4 \\
1 & 1^\circ & 2^\circ
\end{pmatrix*}
\end{align*}

\begin{align*}
\tilde{P}_5 = 
\vy{
\overline{1} & 2^\circ \cr
1^\circ \cr
} \quad
P_5 =
\vy{
2 & 2 \cr
3 \cr}
 \quad
L_5 = 
\begin{pmatrix*}[l]
3 \\
2
\end{pmatrix*} \quad
{\pi}_5 =
\begin{pmatrix*}[l]
4 & 4 & 4 & 4 \\
1 & 1 & 1^\circ & 2^\circ
\end{pmatrix*}
\end{align*}

\begin{align*}
\tilde{P}_4 = 
\vy{
1^\circ & 2^\circ \cr
} \quad
P_4 =
\vy{
2 & 2
\cr} \quad
L_4 = 
\begin{pmatrix*}[l]
3 \\
2
\end{pmatrix*} \quad
{\pi}_4 =
\begin{pmatrix*}[l]
3 & 4 & 4 & 4 & 4 \\
\overline{1} & 1 & 1 & 1^\circ & 2^\circ
\end{pmatrix*}
\end{align*}

\begin{align*}
\tilde{P}_3 = 
\vy{
\overline{1} & 1^\circ & 2^\circ \cr
} \quad
P_3 =
\vy{
2 & 2 & 2 \cr
} \quad
L_3 = \phi \quad
{\pi}_3 =
\begin{pmatrix*}[l]
3 & 3 & 4 & 4 & 4 & 4 \\
1 & \overline{1} & 1 & 1 & 1^\circ & 2^\circ
\end{pmatrix*}
\end{align*}

\begin{align*}
\tilde{P}_2 = 
\vy{
\overline{1} & 1^\circ\cr
} \quad
P_2 =
\vy{
2 & 2 \cr
} \quad
L_2 = \phi \quad
{\pi}_2 =
\begin{pmatrix*}[l]
2 & 3 & 3 & 4 & 4 & 4 & 4 \\
2^\circ & 1 & \overline{1} & 1 & 1 & 1^\circ & 2^\circ
\end{pmatrix*}
\end{align*}

\begin{align*}
\tilde{P}_1 = 
\vy{
\overline{1} \cr}
 \quad
P_1 =
\vy{
2
\cr} \quad
L_1 = \phi \quad
{\pi}_1 =
\begin{pmatrix*}[l]
2 & 2 & 3 & 3 & 4 & 4 & 4 & 4 \\
1^\circ & 2^\circ & 1 & \overline{1} & 1 & 1 & 1^\circ & 2^\circ
\end{pmatrix*}
\end{align*}

\begin{align*}
\tilde{P}_0 = \phi \quad
P_0 = \phi \quad
L_0 = \phi \quad
\pi={\pi}_0 =
\begin{pmatrix*}[l]
2 & 2 & 2 & 3 & 3 & 4 & 4 & 4 & 4 \\
\overline{1} & 1^\circ & 2^\circ & 1 & \overline{1} & 1 & 1 & 1^\circ & 2^\circ
\end{pmatrix*}
\end{align*}

\end{exa}

The final step is to map Burge two-line arrays to semistandard tableaux $Q_\beta $, where $\beta^t$ is even, which is achieved using a correspondence due to Burge \cite{burge}; this correspondence is also used in \cite{sundaram}.\\
\begin{enumerate}
	\item Set $Q_0 = \phi $.
	\item Iterate over the columns of $L$ from left to right, starting with the first column. For the $k^{th} $ column in $L$ with $i$ and $j$ as top and bottom entries respectively, to obtain $Q_k $ from $Q_{k-1} $
	 perform the following steps:
	\begin{enumerate}
		\item Insert $j$ into $Q_{k-1} $ using ordinary row insertion to get  an intermediate tableau $T = Q_{k-1} \leftarrow j $, which has one more box than $Q_{k-1}$.
		\item Add a new box containing $i$ to $T$  directly below the new box that was created by the row insertion process to obtain $Q_k$.
	\end{enumerate} 
	\item The image of $L$ is $Q_\beta =Q_s$, which is a tableau with even length columns.
\end{enumerate}

\begin{exa}

Let $$L = 
\left( \begin{array}{lllll}
 4 & 4 & 5 & 8 & 9\\
 1 & 3 & 2 & 5 & 4
 \end{array}\right).$$  We apply the algorithm to obtain $Q_5$, which is a tableau with even length columns.

$Q_1=\vy{1 \cr 4 \cr}$, \ $Q_2=\vy{1 & 3 \cr 4 & 4 \cr}$, \ $Q_3=\vy{1 & 2 \cr 3 & 4 \cr 4 \cr 5 \cr}$, \ $Q_4=\vy{1 & 2 & 5 \cr 3 & 4 & 8 \cr 4 \cr 5 \cr}$, \ $Q_5=\vy{1 & 2 & 4 \cr 3 & 4 & 5 \cr 4 & 8 \cr 5 & 9 \cr}$
 \end{exa}

\indent The inverse algorithm maps a semistandard tableau $Q_\beta$ of shape $\beta$ with even length columns and $2s$ boxes to a Burge two-line array $L$ via the procedure described below.
\begin{enumerate}
	\item Start by setting $L_s = \phi$ and $Q_s = Q_\beta $.
	\item To obtain $L_{k-1} $ and $Q_{k-1} $ from $L_{k} $ and $Q_{k} $, and starting with $k=n$, find the rightmost box in $Q_k$ that contains its largest entry $i$ and delete it. Perform reverse row bumping on the entry 
	immediately above the removed box to bump an entry $j$ out of the first row.  Let the remaining tableau be $Q_{k-1}$, which has two fewer boxes than $Q_k$ and add the column $\binom{i}{j}$ in $L_{k} $ to obtain $L_{k-1}$. 
	\item Continue to obtain $L = L_0 $.
\end{enumerate}

\begin{exa}
We perform the inverse algorithm on $Q_3=Q=\vy{ 2 & 3 \cr 3 & 6 \cr 4 \cr 6 \cr}.$

$Q_2=\vy{2 \cr 3 \cr 4 \cr 6 \cr}$, $L_2=\left( \begin{array}{c} 6 \cr 3 \cr \end{array} \right)$;  $Q_1=\vy{3 \cr 4 \cr}$, $L_1=\left(\begin{array}{cc} 6 & 6 \cr 2 & 3 \cr \end{array}\right)$;
$L=L_0=\left( \begin{array}{ccc} 4 & 6 & 6 \cr 3 & 2 & 3 \cr \end{array} \right)$

\end{exa}

\bigskip

Let $\overline{X}=\{x_1,x_1^{-1}, \ldots,x_m,x_{m}^{-1}\}$, $\mathcal{T}=\{t_1,\ldots,t_n\}$, $Z=\overline{X}\cup \mathcal{T}$ and $Y=\{y_1,\ldots,y_q\}$.  Given a two-line array $\pi \in \mathcal{A}$, define the weight of $\pi$ as 
$$\mbox{wt}(\pi)=\prod_{1 \leq i \leq q}\prod_{1\leq j \leq m} \prod_{1 \leq k \leq n} y_i^{\alpha_i}x_j^{\beta_j-\beta_{\overline{j}}}t_k^{\gamma_k},$$
where $\alpha_i$ denotes the number of entries $i$ in the top row of $\pi$, $\beta_j$ and $\beta_{\overline{j}}$ the number of entries $j$ and $\overline{j}$ respectively in the bottom row of $\pi$ and $\gamma_k$ the
number of entries equal to $k^\circ \in B_1$ in the bottom row of $\pi$.  For example, the array $\pi$ in Example \ref{firstmap}, has $\mbox{wt}(\pi)=y_2^3y_3^2y_4^4x_1t_1^2t_2^2$.

If a column ${i} \choose {j}$ appears $k$ times in $\pi$, where $j \in B_0$, it contributes either $y_i^kx_i^k$ or $y_i^kx_i^{-k}$ to $\mbox{wt}(\pi)$ and columns ${i} \choose {j}$ where $j \in B_1$ can occur only once in two-line arrays $\pi \in \mathcal{A}$.
With this in mind, the $spo$-correspondence between two-line arrays and triples $(\tilde{P}_\lambda, P_\lambda, Q_\beta)$ yields the following identity:

$$\sum_{\lambda} spo_\lambda(Z) s_\lambda(Y) \sum_{\beta^t even} s_{\beta}(Y) 
=\prod_{\substack{1 \leq i \leq n \\ 1 \leq j \leq q}} (1+t_iy_j) \ds \prod_{\substack{1 \leq i \leq m \\ 1 \leq j \leq q}} (1-x_iy_j)^{-1} (1-x_i^{-1}y_j)^{-1}.$$
The left-hand side of the equation enumerates the triples $(\tilde{P}_\lambda, P_\lambda,Q_{\beta})$ while the right-hand side enumerates the two-line arrays in $\mathcal{A}$.
The Littlewood identity (see \cite{littlewood}, \cite[Exercise 7.28]{stanley} or \cite[I.5.]{macdonald})  states that $\ds \prod_{1 \leq i < j \leq q} (1-y_iy_j)^{-1} = \sum_{\beta^t even} s_{\beta}(Y)$, so rearranging gives
 an orthosymplectic analogue of the Cauchy identity \cite[Theorem 4.24 (c)]{benkart}.
\begin{thm} We have
$$\sum_{\lambda} spo_\lambda(Z) s_\lambda(Y) =   \prod_{\substack{1 \leq i \leq n \\  1 \leq j \leq q}} (1+t_iy_j) \ds \prod_{\substack{1 \leq i \leq m \\ 1 \leq j \leq q}}(1-x_iy_j)^{-1} (1-x_i^{-1}y_j)^{-1} \prod_{1 \leq i < j \leq q} (1-y_iy_j).$$
\end{thm}

\bigskip

\noindent Specializing to $x_i=t_i=1$ yields the following Corollary.

\begin{cor} We have
$$\sum_{\lambda} f_{spo}^{\lambda} s_\lambda(Y)=\prod_{1 \leq j \leq q} (1+y_j)^n(1-y_j)^{-2m} \prod_{1 \leq i<j \leq q} (1-y_iy_k),$$
where $f^\lambda_{spo}$ denotes the number of $spo(2m,n)$-tableaux of shape $\lambda$.\end{cor}

\bigskip

A sequence of partitions $\Lambda=(\emptyset=\lambda^0,\lambda^1, \ldots, \lambda^k=\lambda)$ is an {\em up-down tableau} of shape $\lambda$ and length $k$ if each $\lambda^i$ has one more box or one less box than $\lambda^{i-1}$ for $i=1, \ldots, k$.  An {\em up-down $(m,n)$-tableau} is an up-down tableau $\Lambda$ such that $\lambda_{m+1}^i \leq n$ for each $1 \leq i \leq k$.  An $m$-symplectic up-down tableau $\Lambda$ is an up-down $(m,n)$-tableau with $n=0$.  

The $spo$-insertion algorithm we defined in Section \ref{sec:jdt} can be adjusted to give a bijection between the set of words of length $k$ in the letters $B_0\cup B_1$ and pairs $(T,\Lambda)$ of $spo$-tableaux $T$ of shape $\lambda$, where $\lambda \vdash k-2r$  for some $r$ with $0 \leq r \leq \lfloor k/2 \rfloor$,
and up-down $(m,n)$-tableaux $\Lambda$ of shape $\lambda$ and
length $k$.  The bijection given in \cite{benkart} also yields Corollary \ref{benkupdown}.

Given a word $w=w_1 \ldots w_k$, where  $w_i \in B_0 \cup B_1$, let $P_0=\emptyset$ and $\lambda^0=\emptyset$.  For $1 \leq i \leq k$, let  $P_i=P_{i-1} \leftarrow w_i$ and $\lambda^i$ be the shape of $T_i$.  Then the word $w$ is mapped to $(T,\Lambda)$ where $T=P_k$ and $\Lambda=(\lambda^0,\lambda^1, \ldots, \lambda^k)$.
 We give an example of our bijection for completeness.

\begin{exa}  Let $w=1^\circ \ \ov{1}\  \ \ov{2}\  \ 2^\circ \ 1\  \  2\  \ \ov{1}\  \ 1$.
\begin{align*} P_0=\emptyset \quad \lambda^0=\emptyset; \quad P_1=1^\circ \quad \lambda^1=\vy{\cr} \end{align*}
\begin{align*} P_2=\vy{\ov{1} \cr 1^\circ \cr} \quad \lambda^2=\vy{ \cr \cr}; \quad P_3=\vy{\ov{1} & \ov{2} \cr 1^\circ \cr} \quad \lambda^3=\vy{ & \cr \cr}\end{align*}
\begin{align*} P_4=\vy{\ov{1} & \ov{2} & 2^\circ \cr 1^\circ \cr} \quad \lambda^4=\vy{ & & \cr \cr}; \quad P_5=\vy{\ov{2} & 2^\circ \cr 1^\circ \cr} \quad \lambda^5=\vy{ & \cr \cr} \end{align*}
\begin{align*} P_6=\vy{1 & 2^\circ \cr \ov{2} \cr 1^\circ \cr} \quad \lambda^6=\vy{ & \cr \cr \cr}; \quad
 P_7=\vy{\ov{1} & 2^\circ \cr 1^\circ \cr} \quad \vy{ & \cr \cr} \end{align*}
 \begin{align*} P_8=\vy{1^\circ & 2^\circ \cr} \quad \lambda^8=\vy{ & \cr} \end{align*}
 
 \noindent The word $w$ maps to $(P_8,\Lambda)$, where $\Lambda=(\lambda^0,\lambda^1, \ldots,\lambda^8).$

\end{exa}

\begin{cor}\label{benkupdown} \cite[Corollary 5.6]{benkart} Given positive integers $k,m$ and $n$, we have
$$(2m+n)^k=\sum_{r=0}^{\lfloor k/2 \rfloor} \sum_{\lambda \vdash k-2r }f_{spo}^{\lambda}f^\lambda_{ud}, $$
where $f^\lambda_{ud}$ is the number of up-down $(m,n)$-tableaux of shape $\lambda$ and length $k$.

\end{cor}

There is a bijection between two-line arrays $\pi \in \mathcal{A}$ with $k$ columns and the set of $q \times (2m+n)$ matrices $M_\pi$ with integer entries greater than or equal to zero, where the rightmost $q \times n$ portion of $M_\pi$ contains entries from $\{0,1\}$ and the sum of the
entries in $M_{\pi}$ is equal to $k$.    A two-line array $\pi \in \mathcal{A}$ is associated with a matrix $M_{\pi}$, where the rows of $M_{\pi}$ are labelled by the entries in $\{1,\ldots,q\}$  and the columns by the entries in $B_0 \cup B_1$ in ascending order
where the $ij$-th entry in $M_{\pi}$ is equal to the number of times the column ${i \choose j}$ appears in $\pi$.  The process can clearly be reversed.  Using this bijection to enumerate the two-line arrays $\pi \in \mathcal{A}$ with $k$ columns, coupled with our $spo$-correspondence, yields the following.

\begin{cor}Given a positive integer $k$,
$$\sum_{r=0}^k \sum_{\mu \vdash k-r} {nq \choose r} \frac{ {2mq \choose \ell(\mu)}\ell(\mu)!}{\prod_{i=1}^{k-r} \mbox{cont}(\mu,i)} = \sum_{\vert \lambda \vert +\vert \beta \vert =k} f^\lambda_{spo}d^\lambda d^{\beta}_{even},$$ 
where $d^\lambda$ is the number of semistandard tableaux of shape $\lambda$ with entries in $\{1,\ldots, q\}$, $d^{\beta}_{even}$ is the number of semistandard even tableaux of shape $\beta$, $\mu=(\mu_1,\mu_2,\ldots,\mu_t)$ and $\mbox{cont}(\mu,i)$ denotes the number parts $\mu_j$ such that $\mu_j=i$.\end{cor}
\begin{proof}
Given a two-line array with $k$ columns, if its associated matrix has $r$ entries equal to one in the rightmost $q \times n$ portion of the matrix then the entries in the $q \times 2m$ portion of the matrix must sum to $k-r$ so correspond to a partition $\mu$ of $k-r$.  For each such partition $\mu$ of $k-r$ there are 
$$\frac{2mq}{(2mq-\ell(\mu))!\prod_{i=1}^{k-r} cont(\mu,i)}=\frac{{2mq \choose \ell(\mu)} \ell(\mu)!}{\prod_{i=1}^{k-r} cont(\mu,i)}$$
such matrices.

\end{proof}

\section{Establishing the dual orthosymplectic Cauchy identity}\label{sec:dualortho}

In this section we give a bijective proof of the dual orthosymplectic Cauchy identity  \cite[Theorem 4.24]{benkart}.
\begin{thm}  \label{dualthm} We have
$$ \sum_{\lambda}spo_{\lambda} (Z)s_{\lambda^t}(Y)
= \prod_{\substack{1 \leq i \leq m \\ 1 \leq j \leq q}} (1+x_iy_j)(1+x_i^{-1}y_j) \prod_{\substack{1 \leq i \leq n \\ 1 \leq j \leq q}} (1-t_iy_j)^{-1}
\prod_{1 \leq i \leq j \leq q} (1-y_iy_j).$$

\end{thm}

\bigskip

\noindent Using the dual Burge correspondence \cite[Theorem 5.4]{sundaram}, which states states that $$\prod_{1 \leq i\leq j \leq q} (1-y_iy_j)^{-1}=\sum_{\beta \ even} s_\beta(Y),$$   and rearranging the equation in 
Theorem \ref{dualthm}, we have

$$\displaystyle \sum_{\lambda}spo_{\lambda} (Z)s_{\lambda^t}(Y) \sum_{\beta \ even} s_{\beta}(Y)=\prod_{\substack{1 \leq i \leq m \\ 1 \leq j \leq q}}  (1+x_iy_j)(1+x_i^{-1}y_j) \prod_{\substack{1 \leq i \leq n \\ 1 \leq j \leq q}}
 (1-t_iy_j)^{-1}.$$

\bigskip

The left-hand side enumerates triples $(\tilde{P}_\lambda, P_{\lambda^t}, P_\beta)$, where $\tilde{P}_\lambda$ is an $spo$-tableau of shape $\lambda$, $P_{\lambda^t}$ is a semistandard tableau of shape $\lambda^t$ and $P_\beta$ is a semistandard tableau with even-length rows.  The right-hand side counts two-line arrays $\pi \in \mathcal{A^*}$.  To prove the identity, we will define a bijection between the two sets, which we will call the {\em dual spo-correspondence}.  We begin by defining a bijection from the set $\mathcal{A}^*$ to the set of triples $(\tilde{P}_\lambda, P_{\lambda^t}, L)$, where $L$ is a two-line dual Burge array.  Given $\pi \in \mathcal{A}^*$, perform the following steps:

\begin{enumerate}
  \item Let $\tilde{P}_0 = P_0 = L_0 = \emptyset $.
  \item For $k \geq 1$ and starting from the leftmost column of $\pi$, consider the $k^{th} $ column $\binom{i}{j} $ in $\pi $. Insert $j$ into $\tilde{P}_{k-1} $ via $spo$-insertion to obtain $\tilde{P}_k =\tilde{P}_{k-1}\leftarrow j$. Then $\tilde{P}_k$ has either one more box or one less box than $\tilde{P}_{k-1}$.  \begin{enumerate}
    \item If $\tilde{P}_k$  has one more box than $\tilde{P}_{k-1}$, 
    add a new box with entry $i$ to $P_{k-1}$ at the position where the extra box appears in $\tilde{P}_k $ to form  $P_k $, which has  the same shape as  $\tilde{P}_k $ and let $L_k = L_{k-1} $.
    \item If $\tilde{P}_k$  has one less box than $\tilde{P}_{k-1}$, consider the location of this extra box in $P_{k-1}$ and perform reverse dual row insertion:  bump the entry $x$ in that box into the row above it by replacing the largest entry $x^\prime$ in that row that is less than or equal to $x$.  Bump $x^\prime$ into the row above it using the same procedure and continue until an entry $a$ is bumped out of the first row.
     Let $P_k$ be the resulting tableau, which has the same shape as $\tilde{P}_k $ and add the column $\binom{i}{a} $ on the right of $L_{k-1} $ to obtain $L_k $.
  \end{enumerate}
  \item Let $\tilde{P}_\lambda = \tilde{P}_s $ and $P_{\lambda^t} = (P_s)^t $.  The array $L_s$ satisfies the property that whenever entries $i_k=i_{k+1}$ in its top row the corresponding bottom entries $b_k$ and $b_{k+1}$ satisfy
  $b_k \leq b_{k+1}$ (see proof of Theorem \ref{maindualthm}).  Sort the bottom row of $L_n$ so that whenever there are two entries $i_k=i_{k+1}$ in the top row, the corresponding entries $b_k$ and $b_{k+1}$ in the bottom row 
  satisfy $b_k \geq b_{k+1}$ and let $L$ be the resulting array.
\end{enumerate}

\begin{exa}
$$
\text{Let } \pi = 
\left(\begin{array}{llllllllllllllll}
1 & 1 & 1 & 2 & 2 & 2 & 3 & 3 & 3 & 3 & 3 & 4 & 4 & 4 & 4 & 4 \cr
 \overline{5} & \ov{3} & 1 & 3^\circ & 5 & \ov{4} & 1^\circ & 1 ^\circ & \ov{2} & 2 & \ov{1} & 4 & \ov{3} & 3 & \ov{2} & 1 \cr
\end{array}\right) $$
We omit the first  several steps.

\begin{align*}
\tilde{P}_9 = \vy{\ov{2} & \ov{4} & 1^\circ \cr 3 & 5 & 1^\circ \cr \ov{3} & 3^\circ \cr \ov{5} \cr}
 \qquad P_{9} =\vy{ 1 & 2 & 3 \cr 1 & 2 & 3 \cr 1 & 2 \cr 3 \cr} \qquad L_9=\emptyset
\end{align*}

\begin{align*}
\tilde{P}_{10} = 
\vy{2 & \ov{4} & 1^\circ \cr \ov{2} & 5 & 1^\circ \cr \ov{5} & 3^\circ \cr} 
\qquad P_{10} =
\vy{1 & 2 & 3 \cr 1 & 2 & 3 \cr 1 & 3 \cr}
\qquad
 L_{10} = \begin{pmatrix*}[l]
3\\
2
\end{pmatrix*}
\end{align*}

\begin{align*}
\tilde{P}_{11} = 
\vy{\ov{1} & \ov{4} & 1^\circ \cr 5 & 1^\circ \cr \ov{5} & 3^\circ \cr} 
\qquad P_{11} =
\vy{1 & 2 & 3 \cr 1 & 2 \cr 1 & 3 \cr}
\qquad
 L_{11} = \begin{pmatrix*}[l]
3 & 3 \cr
2 & 3 \cr
\end{pmatrix*}
\end{align*}

\begin{align*}
\tilde{P}_{12} = 
\vy{\ov{1} & 4 & 1^\circ \cr \ov{4} & 1^\circ \cr 5 & 3^\circ \cr \ov{5} \cr} 
\qquad P_{12} =
\vy{1 & 2 & 3 \cr 1 & 2 \cr 1 & 3 \cr 4 \cr}
\qquad
 L_{12} = \begin{pmatrix*}[l]
3 & 3 \cr
2 & 3 \cr
\end{pmatrix*}
\end{align*}

\begin{align*}
\tilde{P}_{13} = 
\vy{\ov{1} & \ov{3} & 1^\circ \cr 4 & 1^\circ \cr \ov{4} & 3^\circ \cr 5 \cr \ov{5} \cr} 
\qquad P_{13} =
\vy{1 & 2 & 3 \cr 1 & 2 \cr 1 & 3 \cr 4 \cr 4 \cr}
\qquad
 L_{13} = \begin{pmatrix*}[l]
3 & 3 \cr
2 & 3 \cr
\end{pmatrix*}
\end{align*}

\begin{align*}
\tilde{P}_{14} = 
\vy{\ov{1} & 3 & 1^\circ \cr \ov{3} & 1^\circ \cr 4 & 3^\circ \cr \ov{4} \cr} 
\qquad P_{14} =
\vy{1 & 2 & 3 \cr 1 & 3 \cr 1 & 4 \cr 4 \cr}
\qquad
 L_{14} = \begin{pmatrix*}[l]
3 & 3 & 4 \cr
2 & 3 & 2 \cr
\end{pmatrix*}
\end{align*}

\begin{align*}
\tilde{P}_{15} = 
\vy{\ov{1} & \ov{2} & 1^\circ \cr 3 & 1^\circ \cr \ov{3} & 3^\circ \cr} 
\qquad P_{15} =
\vy{1 & 2 & 3 \cr 1 & 4 \cr 1 & 4 \cr}
\qquad
 L_{15} = \begin{pmatrix*}[l]
3& 3 & 4 & 4  \cr
2 & 3 & 2 & 3 \cr
\end{pmatrix*}
\end{align*}

\begin{align*}
\tilde{P}_{16} = 
\vy{\ov{2} &  1^\circ \cr 3 & 1^\circ \cr \ov{3} & 3^\circ \cr} 
\qquad P_{16} =
\vy{1 & 2 \cr 1 & 4 \cr 1 & 4 \cr}
\qquad
 L_{16} = \begin{pmatrix*}[l]
3 & 3 & 4 & 4 & 4 \cr
2 & 3 & 2 & 3 & 3 \cr
\end{pmatrix*}
\end{align*}

The image of $\pi$ is $(\tilde{P}_{16}, P_{16}^t, L) $, where $L=\begin{pmatrix*}[l]
3 & 3 & 4 & 4 & 4 \cr
3 & 2 & 3 & 3 & 2 \cr
\end{pmatrix*}$.
\end{exa}

\bigskip

To map a dual Burge array to a semistandard tableau $Q_\beta$, where $\beta^t$ has rows of even length, we use the dual Burge correspondence \cite{burge}.  This is similar to the Burge correspondence described earlier, except that instead of performing regular row insertion in step 2. a), we perform dual row insertion.  The new box from step 2. b) is added directly to the right of the newly created box and not directly below it like previously.\\

\begin{exa}  Consider the dual Burge array $L=\begin{pmatrix*}[l]
3 & 3 & 4 & 4 & 4 \cr
3 & 2 & 3 & 3 & 2 \cr
\end{pmatrix*}.$  The dual Burge correspondence yields the tableau $S$, where $S^t=\vy{2 & 2 & 3 & 3 & 3 & 3 & 4 & 4 \cr 3 & 4 \cr}.$

\end{exa}

\bigskip

The following is a description of the inverse dual algorithm that maps a triple $(\tilde{P}_\lambda, P_{\lambda^t}, L) $ to a two-line array $\pi \in \mathcal{A}^* $, where $\tilde{P}_\lambda $ is an $spo$-tableau of shape $\lambda$, $P_{\lambda^t} $ is a semistandard tableau of shape $\lambda^t$ and $L$ is a dual Burge two-line array.\\

\begin{enumerate}
  \item Let $\tilde{P}_s = \tilde{P}_\lambda $ and $P_s = (P_{\lambda^t})^t $.  Rewrite $L_n$ so that when entries $i_k$ and $i_{k+1}$ in its top row satisfy $i_k=i_{k+1}$, the corresponding bottom entries satisfy $b_k \leq b_{k+1}$ and let this be $L_s$.  Let $\pi_s = \emptyset $, where $s$ is the total number of entries in $P_{\lambda^t} $ and $L$.
  \item To obtain $(\tilde{P}_{k-1}, P_{k-1}, L_{k-1} )$ and $\pi_{k-1} $ from $(\tilde{P}_{k}, P_{k}, L_{k} )$ and $\pi_{k} $, find the largest entry $i$ among all the entries in $P_k$ and the top row of $L_k$.
  \begin{enumerate}
    \item If $i$ belongs to $L_k$,  
   select the rightmost column $\binom{i}{a} $ with $a$ maximal and delete this column from $L_k$ to obtain $L_{k-1}$. Use dual row insertion to insert $a$ into $P_k$
     to obtain $P_{k-1}$, which has one more box than $P_k$. 
     The tableau $P_{k-1} $ has one more box than $\tilde{P}_k$; add an empty box to $\tilde{P}_k $ in this spot and perform reverse jeu de taquin to slide the empty box into the first column and into the highest possible row $r$, 
     without creating a symplectic violation. Place an $r$ in the empty box and then replace the right-most $r$ in that row with $\overline{r} $. (Note that there is at least one $r$ in the row $r$, since we placed an $r$ in the empty
     box.)  Then bump an entry $r$ into the row above and perform reverse row bumping until some entry $j$ is bumped out of the first row to give $P_{k-1} $. 
    Prepend column $\binom{i}{j} $ to $\pi_k $ to obtain $\pi_{k-1} $.
    \item If $i$ does not belong to $L_k$ then select the lowest entry $i$ in $P_k$ and delete that box to get $P_{k-1} $. Since $\tilde{P}_k $ and $P_k $ are of the same shape, $\tilde{P}_k $ and $P_{k-1} $
     differ by exactly one box.  Locate the extra box in this location in $\tilde{P}_k $ and perform reverse $spo$-insertion, starting with the entry in this box, to obtain $\tilde{P}_{k-1} $. This bumps out some entry $j$ from the first row. 
     Prepend column $\binom{i}{j} $ to $\pi_k $ to get $\pi_{k-1} $ and let $L_{k-1} = L_k$.
  \end{enumerate}
  \item Set $\pi = \pi_0 $.
\end{enumerate}

\begin{lem}\label{duallemma1} Let $T$ be an $spo$-tableau and $x,x^\prime \in B_0 \cup B_1$.  Suppose $x^\prime \geq x$ if $x,x^\prime \in B_1$ and $x^\prime > x$ otherwise.  If $T \leftarrow x^\prime$ has one more box $B^\prime$
than $T$ and $(T \leftarrow x^\prime)\leftarrow x$ has one more box $B$ than $T \leftarrow x^\prime$ then $B$ lies strictly below $B^\prime$.
\end{lem}
\begin{proof}
Suppose that $x^\prime$ bumps $y^\prime$ from row one of $T$ and that $x$ subsequently bumps $y$ from row one.  Then $y$ either belongs to the same box as $y^\prime$ or to a box to its left.  The two bumped entries $y$ and $y^\prime$ can only be equal if $y,y^\prime \in B_0$ and otherwise $y^\prime>y$ so the argument can also be applied to all rows below the first.  The bumping route for $x^\prime$ ends with the insertion of an entry $z^\prime$ in the box $B^\prime$ and the bumping route for $x$ then either bumps $z^\prime$ into the next row or bumps an entry to the left of $B^\prime$ into the next row, so the added box $B$ lies strictly below $B^\prime$.  \end{proof}

\begin{lem}\label{jdtlem1} Suppose that $x,x^\prime \in B_0$ with $x^\prime>x$.  If $spo$-insertion of  $x^\prime$ into $T$ causes a cancellation then insertion of $x$ into $T \leftarrow x^\prime$ also causes a cancellation.  Furthermore  the jeu de taquin path of the empty box caused by inserting $x^\prime$ ends in a lower row than that of the empty box caused by inserting $x$.

\end{lem}

\begin{proof}
Since insertion of $x^\prime$ causes a cancellation, it causes some $i \geq x^\prime$ in the first column of $T$ to bump an $\overline{i}$ in row $i$ of $T$ into row $i+1$.  Thus each row $r$ of the first column of $T \leftarrow x^\prime$ with $r \leq i-1$ contains either  $r$ or $\overline{r}$.   The bumping path of $x$  in the symplectic part of the  tableau lies weakly left and strictly above the bumping path of $x^\prime$ (see \cite{fulton}) so  insertion of $x$ causes an entry to bump out of one of the first $i-1$ rows of  the first column of $T \leftarrow x^\prime$ into a subsequent row, which causes a cancellation in the first column.

Suppose that the empty box caused by inserting $x^\prime$  enters row $i+1$ at column $m$ and let $a_m$ denote the entry in row $i$ of column $m$ of $T$ and $b_m$ the entry directly above it. Then this box slides $a_m$ into column $m-1$ of row $i$ and, if the subsequent box formed by inserting $x$ enters row $i$, it must slide to at least column $m$ or a column to the right of column $m$ before entering row $i+1$ since $a_m\geq b_m$ with possible equality if $a_m, b_m \in B_1$. The same argument applies to subsequent rows and the result now follows. \end{proof}

\begin{thm}\label{maindualthm} The triple $(\tilde{P}_\lambda, P_\lambda^t,L)$ consists of an $spo$-tableau $\tilde{P}_\lambda$, a semistandard tableau $P_\lambda^t$ and a dual Burge array $L$. \end{thm}

\begin{proof}  Since the entries in the top row of $\pi\in \mathcal{A}^*$ are weakly increasing and boxes added at any step to $P_{k-1}$ to obtain $P_k$ are in outer corners, the entries in the rows and columns of any $P_k$ in the algorithm are weakly increasing.  
For entries $a_i=a_{i+1}$ in the top row of $\pi \in \mathcal{A}^*$, we have $b_i \geq b_{i+1}$ and Lemma \ref{duallemma1} guarantees that a box added by inserting $b_i$ into $\tilde{P}_{i-1}$ is strictly above a box added by inserting $b_{i+1}$ into
$\tilde{P}_i$, so the rows of any $P_k$ obtained in the algorithm are strictly increasing, and thus $P_\lambda^t$ is semistandard.

If $L_n=\begin{pmatrix*}[l] i_1 & i_2 & \cdots & i_r \cr j_1 & j_2 & \cdots & j_r \cr \end{pmatrix*},$  the justification that $i_k \leq i_{k+1}$ and that $i_k \geq j_k$ is the same as in the proof of Theorem \ref{bijectionthm}.  If $i_k=i_{k+1}$, then the entries $b^\prime$ and $b$ below $i_k$ and $i_{k+1}$ in $\pi$ respectively satisfy $b^\prime > b$ and both cause cancellations.  By Lemma \ref{jdtlem1}  the taquin path of the empty box created by inserting $b^\prime$ ends in a lower row  than the taquin path of the empty box created by subsequently inserting $b$.  It follows that the corresponding reverse row bumping path in $P_{k}$, which bumps $j_{k}$ out of the first row, begins in a lower row than the bumping path that bumps $j_{k+1}$ out of the first row of $P_{k+1}$ and, thus,  $j_k \leq j_{k+1}$. After rearranging, the array $L$ is a dual Burge array.  \end{proof}

\noindent {\bf Acknowledgement.}  The authors wish to thank the anonymous referees who carefully read the paper and provided useful suggestions.

\end{document}